\documentclass[reqno, 11pt, a4paper]{amsart} % removed by Nina
\usepackage{color}
\usepackage{graphicx}
\usepackage{amsmath,amsfonts,amssymb,graphics,amsthm}
\usepackage{hyperref}
\usepackage{comment}
\usepackage{tabularx}
\usepackage{enumerate}
\usepackage{bbm}
\usepackage{mathrsfs}
\usepackage[margin=1.2in]{geometry}

\usepackage{latexsym,ifthen,xspace} % removed by Nina
\usepackage{enumerate, calc} % removed by Nina
\newtheorem{theorem}{Theorem}[section]
\newtheorem{lemma}[theorem]{Lemma}

\theoremstyle{definition}

\theoremstyle{remark}

\numberwithin{equation}{section}

\DeclareMathAlphabet{\mathsl}{OT1}{cmss}{m}{sl}
\SetMathAlphabet{\mathsl}{bold}{OT1}{cmss}{bx}{sl}

\def\TH(#1){\label{#1}}\def\thv(#1){\ref{#1}}
\def\Eq(#1){\label{#1}}\def\eqv(#1){(\ref{#1})}

\newcommand{\D}{\mathbbm{D}}
\newcommand{\E}{\mathbbm{E}}
\newcommand{\N}{\mathbbm{N}}

\newcommand{\R}{\mathbbm{R}}
\renewcommand{\P}{\mathbbm{P}}

\newcommand{\eps}{\varepsilon}

\DeclareMathOperator{\Var}{Var}

\newcommand{\aryb}{\begin{eqnarray*}}
\newcommand{\arye}{\end{eqnarray*}}
\def\alb#1\ale{\begin{align*}#1\end{align*}}
\newcommand{\eqb}{\begin{equation}}
\newcommand{\eqe}{\end{equation}}
\newcommand{\eqbn}{\begin{equation*}}
\newcommand{\eqen}{\end{equation*}}

\newcommand{\BB}{\mathbbm}

\newcommand{\op}{\operatorname}
\newcommand{\rta}{\rightarrow}

\newcommand{\wt}{\widetilde}
\newcommand{\wh}{\widehat}

\def\y{{\bf y}}

\begin{document}

\title[Trace reconstruction with varying deletion probabilities]{Trace reconstruction with varying deletion probabilities}

%\author{
%	Lisa Hartung
%	\thanks{Courant, New York University; \texttt{hartung@cims.nyu.edu}}
%	\and
%	Nina Holden
%	\thanks{Massachusetts Institute of Technology; \texttt{ninah@math.mit.edu}}
%	\and
%	Yuval Peres
%	\thanks{Microsoft Research; \texttt{peres@microsoft.com}}
%}

\author{Lisa Hartung}
\address{Courant Institute of Mathematical Sciences\\ New York University, New York, NY 10012} \email{hartung@cims.nyu.edu}

\author{Nina Holden}
\address{Department of Mathematics\\
	Massachusetts Institute of Technology\\
	Cambridge, MA 02139} \email{ninah@mit.edu}

\author{Yuval Peres}
\address{Microsoft Research\\
	Redmond, WA 98052} 
\email{peres@microsoft.com}

%    Remove any unused author tags.

%    author one information
 \date{\today}

\begin{abstract}
	In the trace reconstruction problem an unknown string ${\bf x}=(x_0,\dots,x_{n-1})\in\{0,1,...,m-1\}^n$ is observed through the deletion channel, which deletes each $x_k$ with a certain probability, yielding a contracted string $\wt {\bf X}$. Earlier works have proved that if each $x_k$ is deleted with the same probability $q\in[0,1)$, then $\exp(O(n^{1/3}))$ independent copies of the contracted string $\wt {\bf X}$ suffice to reconstruct ${\bf x}$ with high probability. We extend this upper bound to the setting where the deletion probabilities vary, assuming certain regularity conditions. First we consider the case where $x_k$ is deleted with some known probability $q_k$. Then we consider the case where each letter $\zeta\in \{0,1,...,m-1\}$ is associated with some possibly unknown deletion probability $q_\zeta$.
\end{abstract}

\maketitle

\section{Introduction}
Let $n\in\N$, $m\in\{2,3,\dots \}$, and $[m]=\{0,\dots,m-1 \}$.
In trace reconstruction the goal is to reconstruct an unknown string ${\bf x}=(x_0,\dots,x_{n-1}) \in [m]^n$ from noisy observations of ${\bf x}$. Here we study the case where the data is noisy due to a deletion channel in which each bit is deleted independently with a certain probability. In other words, instead of observing ${\bf x}$ we observe many independent strings ${\bf\widetilde X}$ obtained by sending $\bf x$ through the deletion channel. The probability of deleting $x_k$ might depend on either (I) the location $k$ of $x_k$ in the string, or (II) the letter $x_k$. 

In Case (I) and given $p_k\in(0,1]$ for $k\in[n]$, the string ${\bf\widetilde X}$ is obtained in the following way for $k=0, 1, \dots, n-1$, starting from an empty string.
\begin{itemize}
\item (retention) With probability $p_k$, copy $x_k$ to the end of ${\bf \widetilde X}$ and increase $k$ by one.
\item (deletion) With probability $q_k=1-p_k$, increase $k$ by one.
\end{itemize}
In Case (II) and given $p_\zeta\in(0,1]$ for $\zeta\in[m]$, the string ${\bf\widetilde X}$ is obtained by performing the following steps for $k=0, 1, \dots, n-1$.
\begin{itemize}
	\item (retention) With probability $p_{x_k}$, copy $x_k$ to the end of ${\bf \widetilde X}$ and increase $k$ by one.
	\item (deletion) With probability $q_{x_k}=1-p_{x_k}$, increase $k$ by one.
\end{itemize}
In Case (I) we assume the deletion probabilities are known, while we do not make this assumption in Case (II).

For $T\in\N$ we consider $T$ independent outputs  ${\bf{\widetilde X}}^{(1)},\dots,{\bf{\widetilde X}}^{(T)}$ (called ``traces'') from the deletion channel. Our main question is the following. Given $\eps>0$, how many samples are needed, such that for some ${\bf\widehat X}\in[m]^n$ we have $\P[{\bf\widehat X}={\bf x} ]>1-\eps$? Holenstein, Mitzenmacher, Panigrahy, and Wieder \cite{HMPW08} proved that $\exp\left(n^{1/2}\op{polylog}(n)\right)$ traces are sufficient in the case where all $x_k$ are deleted independently with the same probability $q\in[0,1)$. See \cite{mpv14} for an alternative proof. This result was recently improved by De, O'Donnell, and Servedio \cite{DOS16}, and by Nazarov and Peres \cite{NaPe16}. Using single bit statistics for the traces, they proved that reconstruction is possible with $\exp\left(O(n^{1/3})\right)$ traces, and that this is optimal for reconstruction techniques using only single bit statistics. The case of random strings was studied in \cite{BKKM04} and \cite{PZ17}. 

\subsection{Main result}
We study the following two settings in Case (I) where the deletion probabilities depend on the location in the string.
\begin{itemize}
\item[(i)] (weak monotonicity) For some $\delta\in(0,1)$ the retention probabilities $(p_k)_{k\in \mathbb{N}}$ satisfy $p_\ell>\frac{p_k}{2}+\delta$  for all $k>\ell$, and $p_k>\delta$ for all $k>0$. 
\item[(ii)] (periodicity) The deletion probabilities are $2$-periodic, meaning that 
\begin{equation}\label{cond.3}
q_k=q \mbox{ for }k \mbox{ even}, \quad q_k= \wt q \mbox{ for }k \mbox{ odd}.
\end{equation}
\end{itemize}
In particular, (i) covers the case where the retention probabilities are monotonically decreasing and bounded away from zero, and the case where all retention probabilities are in some interval $[p+\delta,2p]$ for $p\in(0,1/2]$ and $\delta>0$. Also observe that, by reversing the sequence, we can also study strings where the deletion probabilities satisfy $p_\ell>\frac{p_j}{2}+\delta$  for all $\ell>j$ (instead of $\ell<j$).

\begin{theorem}\label{thm.main}
	Let $\eps>0$ and $m\in\{2,3,\dots \}$, and let the deletion probabilities be known and satisfy Assumption (i) or (ii).	There exists a constant $C>0$ depending only on $\eps,m$, and $\delta$ (for Assumption (i)), and $\eps,m,q,$ and $\wt q$ (for Assumption (ii)), such that the original string ${\bf x}\in[m]^n$, can be reconstructed with probability at least $1-\eps$ from $T=\lceil\exp\left(Cn^{1/3}\right)\rceil$ i.i.d.\ samples of the deletion channel applied to ${\bf x}$.
\end{theorem}

For case (II), where the deletion probabilities vary by letter, we prove the following.
\begin{theorem}
	For any $\delta,\eps>0$ and $m\in\{2,3,\dots \}$ there is a constant $C>0$ such that if $\min_{\zeta\in[m]} p_\zeta\geq\delta$, then the string ${\bf x}\in[m]^n$ can be reconstructed with probability at least $1-\eps$ from $T=\lceil\exp\left(Cn^{1/3}\right)\rceil$ i.i.d.\ samples of the deletion channel applied to $\bf x$. This result holds also in the case when the deletion probabilities are unknown.
	\label{thm1}
\end{theorem}

The deletion channel (or variants of it which also allow insertions, substitutions, swaps, etc.) is relevant for the study of mutations and DNA sequencing errors. In the case of mutations several studies have revealed that mutation probabilities vary by location in the genome \cite{esw03,mmh04,swe02,sl15}.

{\bf Outline of the paper:} To prove Theorem \ref{thm.main} we will first derive in Section \ref{sec1} an exact formula expressing the single bit statistics of $\wt{\bf X}$ as the coefficients of a particular polynomial. In Sections \ref{sec2} and \ref{sec3} we prove that this polynomial cannot be too small everywhere on a particular boundary arc of the unit disk. In Section \ref{sec4} we use this to prove that for any two input strings, at least one bit in the trace has a sufficiently different expectation so we are able distinguish the two strings. Theorem \ref{thm1} will be proved by first using the traces to obtain good estimates for the deletion probabilities, and then use a single bit test (with the estimated probabilities) for the traces sent through a second deletion channel.

{\bf Acknowledgements:} We thank Ryan O'Donnell for suggesting the problem of varying deletion probabilities. Most of this work was done while the first and second author were visiting Microsoft Research, and they want to thank Microsoft for the hospitality.

 \section{Preparatory Lemmas}
 
 \subsection{A polynomial identity}
 \label{sec1}

 For ${\bf a }\in\R^n$ define
\begin{equation}
\Psi(w)= \sum_{k=0}^{n-1}a_kp_k \prod_{\ell=0}^{k-1}(p_\ell w+q_\ell).
\end{equation}

 \begin{lemma}\label{lem.1}Let ${\bf{a}}=(a_0,a_1,\dots,a_{n-1})\in\mathbb{R}^n$ and let $\wt{\bf a}$ be the output of the deletion channel with input $\bf{a}$, padded with an infinite sequence of 0's on the right side. Then
 \begin{equation}\label{e.1}
 \mathbb{E}\left(\sum_{j\geq 0} \wt a_j w^j\right)=\Psi(w).
 \end{equation}
 \end{lemma}
 \begin{proof}
 The expectation on the left side of \eqref{e.1} can be written as 
\begin{equation}\label{e.2}
 \sum_{k=0}^{n-1}a_kp_k \sum_{j=0}^{k-1} w^j \mathbb{P}\left(\sum_{\ell=0}^{k-1}B_\ell=j\right),
 \end{equation}
 where $B_\ell$ for $\ell=0,\dots, n-1$ are independent Bernoulli$(p_\ell)$-distributed random variables. 
 Observe that
 \begin{equation}\label{e.3}
 \sum_{j=0}^{k-1} w^j \mathbb{P}\left(\sum_{\ell=0}^{k-1}B_\ell=j\right)
 =
 \mathbb{E}\left(w^{\sum_{\ell=0}^{k-1}B_\ell}\right)
 =
 \prod_{\ell=0}^{k-1}(q_\ell+p_\ell w),
 \end{equation}
 where we used independence in the last step. Plugging \eqref{e.3} into \eqref{e.2} concludes the proof of Lemma \ref{lem.1}.
 \end{proof}
  
 Next, we will establish that (perhaps after a change of variables) the function $\vert\Psi(w)\vert$ on a small arc of the unit disc is not to small. We use tools from standard complex analysis (inspired by \cite{BoEr97}).

\subsection{The deletion probabilities satisfy a monotonicity property}
\label{sec2}
 \begin{lemma}\label{lem.2}
 	Assume that $a_0=1$, $a_k\in[-1,1]$ for $k\geq 1$, and that there is some $\delta\in(0,1/10)$ such that $p_0>\frac{p_k}{2}+\delta$ and $p_k>\delta$ for all $k>0$. Define the arc $\gamma_L=\{z=e^{i\theta}, \frac{-\pi}{L}\leq \theta \leq \frac{\pi}{L} \}$.  Then there is a constant $c>0$ depending only on $\delta$ such that $\max_{\gamma_L}\vert \Psi(w)\vert > e^{-cL}$ for all $L\geq 1$.
 \end{lemma}
 \begin{proof}
 	Throughout the proof $c$ is a positive constant which may depend on $\delta$ and which may change line by line. Define $w_0:=-q_0/p_0$. Let $\Omega'$ be a rectangular domain which contains $w_0$ and $1/2$, has sides parallel to the coordinate axes, and which is such that $|pw+(1-p)|<1-\delta/2$ for all $p\in \{p_0,p_1,\dots \} $ and ${w\in\Omega'}$. Observe that an appropriate domain exists since the required inequality is satisfied for all $w\in[w_0,1/2]$. Furthermore, observe that we may assume $\Omega'$ changes continuously as we vary $w_0$.
 	Let $\Omega=\Omega'\cup B_{1/2}(1/2)$. For $L>100$, let $\Omega'_L$ be the horizontal translation of $\Omega$ by some $t_L\in(0,1/10)$, such that $\gamma_L$ is the intersection of $\Omega'_L$ and the right half of $\partial\D$. There is an $L_0>0$ depending only on $\delta$ such that $w_0\in\Omega'_L$ for all $L>L_0$, and it is sufficient to prove the result of the lemma for $L>L_0$. Let $\Omega_L=\Omega'_L\cap(\D\cup\{z\,:\,\op{Re}(z)<0 \})$. Then $\partial\Omega_L$ is the union of the three disjoint sets $\gamma_L,\gamma_L^2$, and $\gamma_L^3$, where
 	\eqbn
 	\begin{split}
 		\gamma_L^2:=\{z \in \partial\Omega_L \,:\,\op{Re}(z)\leq 1/2+t_L \},\quad
 		\gamma_L^3:=\{z \in \partial\Omega_L\setminus\gamma_L \,:\,\op{Re}(z)>1/2+t_L \}.
 	\end{split}
 	\eqen
 	\begin{figure}
 		\centering
 		\includegraphics[scale=1]{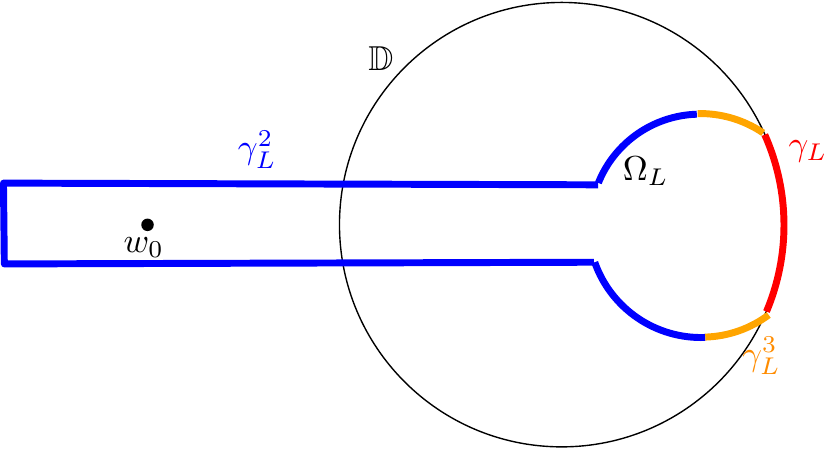}
 		\caption{Illustration of objects defined in the proof of Lemma \ref{lem.2}. The domain $\Omega_L$ is bounded by the colored curves. }
 		\label{fig1}
 	\end{figure}
 See Figure \ref{fig1} for an illustration of $\partial\Omega_L$.	Since $\log|\Psi|$ is subharmonic, and letting $\mu_{\Omega_L}^{w_0}$ denote harmonic measure of $\Omega_L$ relative to $w_0$,
 	\begin{equation}\label{eq2}
 	\log\vert \Psi(w_0)\vert 
 	\leq 
 	\int_{\gamma_L} \log\vert\Psi(z)\vert d\mu_{\Omega_L}^{w_0}(z)
 	+ 
 	\int_{\gamma^2_L} \log\vert\Psi(z)\vert d\mu_{\Omega_L}^{w_0}(z)
 	+ 
 	\int_{\gamma^3_L} \log\vert\Psi(z)\vert d\mu_{\Omega_L}^{w_0}(z).
 	\end{equation}
 	Without loss of generality we assume $|\Psi(z)|<1$ everywhere on $\gamma_L$. For each fixed $w_0$ there is a constant $c_{w_0}$ such that $\mu_{\Omega_L}^{w_0}(\gamma_L) \geq c_{w_0}/L$. Since we assume $\Omega'$ varies continuously as we vary $w_0\in[1-1/\delta,0]$, we can find a $c'>0$ such that $c_{w_0}>c'$ for all $w_0$ this gives
 	\eqbn
 	\begin{split}
 		\int_{\gamma_L} \log\vert\Psi(z)\vert d\mu_{\Omega_L}^{w_0}(z)
 		&\leq 
 		\mu_{\Omega_L}^{w_0}(\gamma_L)\cdot
 		\log\max_{\gamma_L}|\Psi(w)|
 		\leq \frac{c'}{L}\log\max_{\gamma_L}|\Psi(w)|.
 	\end{split}
 	\eqen 
 	Observe that $\log\vert\Psi(z)\vert<c$ for any $z\in \partial \Omega_L\cap\{z\,:\,\op{Re}(z)<1/2+t_L \}$, so
 	\eqbn
 	\int_{\gamma_L^2  } \log\vert\Psi(z)\vert d\mu_{\Omega_L}^{w_0}(z)
 	\leq c.
 	\eqen
 	For any $w\in\D$, 
 	\begin{equation}\label{e.10}
 	\vert\Psi(w)\vert \leq \sum_{k=0}^{n-1} p_k\prod_{\ell=0}^{k-1} \vert q_\ell+p_\ell w\vert
 	\leq  \sum_{k=0}^{n-1}\prod_{\ell=0}^{k-1} (q_k+p_k \vert w\vert)
 	\leq \sum_{k=0}^{n-1}(1-\delta+\delta \vert w\vert)^k 
 	\leq \frac{1}{\delta(1-\vert w\vert)}.
 	\end{equation}
 	It follows that 
 	\begin{equation}\label{eq12}
 		\int_{\gamma_L^3  } \log\vert\Psi(z)\vert d\mu_{\Omega_L}^{w_0}(z) 
 		\leq
 		\log\left(\frac{1}{\delta}\right)
 		+
 		\int_{\gamma^3_L} \log\frac{1}{1-\vert z\vert} d\mu_{\Omega_L}^{w_0}(z).
 	\end{equation}
 	A Brownian motion started at $w_0$ must hit the line segment $\ell_L:=\{ z\in\Omega_L\,:\,\op{Re}(z)=1/4+t_L \}$ before hitting $\gamma_L^3$. Therefore, defining $D_L=B_{1/2}(1/2+t_L)$, the integral in \eqref{eq12} is bounded above by
 	\begin{equation}\label{e.13}
 	c\sup_{w\in\ell_L}\int_{\gamma^3_L} \log\frac{1}{1-\vert z\vert} d\mu_{D_L}^{w}(z).
 	\end{equation}
 	Since $\ell_L$ has distance $>1/4$ from $\gamma_L^3$, the density of $\mu_{D_L}^{w}$ on $\gamma_L^3$ is bounded above by a universal constant. For any $z\in \gamma_L^3$ in the first quadrant we have
 	 $1-\vert z\vert\geq c^{-1} (\theta-\theta_0)^2$, where $\theta\in(-\pi,\pi]$ is the angle for polar coordinates with origin at $1/2+t_L$ and $\theta_0\in[0,\pi/2]$ corresponds to the point of intersection between $\partial D_L$ and $\partial\D$ in the first quadrant (see \cite[equation (4.4)]{NaPe16}). Therefore we can bound \eqref{e.13} from above by
 	\begin{equation}\label{e.16}
 	c\int_{0}^{\pi/2} \log \left \vert \frac{1}{\theta} \right \vert \,d\theta
 	<
 	\infty.
 	\end{equation}
 	We conclude that 
 	\begin{equation}\label{e.12}
 	\int_{\gamma_L^3  } \log\vert\Psi(z)\vert d\mu_{\Omega_L}^{w_0}(z) 
 	<
 	c.
 	\end{equation}
 	Inserting the above estimates into \eqref{eq2}, we get
 	\eqbn
 	\log|p_0| \leq c + \frac{c'}{L} \log\max_{\gamma_L}|\Psi(w)|,
 	\eqen
 	which implies the lemma.
 \end{proof}
 
\subsection{The $2$-periodic case}
\label{sec3}
    Assume without loss of generality that $p<\wt p$. If $z=\wt p w+\wt q$ then
    \begin{equation}\label{eq.per.1}
    \Psi(w)=\wt \Psi(z):= \sum_{k=0}^{n-1}a_kp_k \prod_{\ell=0}^{k-1} (p_\ell \frac{z-\wt q}{\wt p}+q_\ell).
    \end{equation}
    \begin{lemma}\label{lem.4}
  Assume that $a_0=1$, $a_k\in[-1,1]$ for $k\geq 1$, and that $\left(p_\ell\right)_\ell$ is $2$-periodic. If $\vert \wt\Psi(z)\vert<\beta$ everywhere on the the $arc\; \gamma_L=\{z=e^{i\theta}, \frac{-\pi}{L}\leq \theta \leq \frac{\pi}{L} \}$ then there exists $c$ depending on $q$ and $\wt q$ such that $\beta > e^{-cL}$ for all $L\geq 1$.
  \end{lemma}
\begin{proof}
We consider the following cases.
\begin{itemize}
\item[ Case 1)] $p_0=\wt p$.

Then \eqref{eq.per.1} gives
 \begin{equation}\label{eq.per.3}
 \wt \Psi(z)= a_0\wt p+ a_1p z+ a_2\wt p z\left(p\frac{z-\wt q}{\wt p}+q\right)+\dots.
    \end{equation}
    Setting $z=0$ we get $ \wt \Psi(0)=\wt p>0$.
  \item[ Case 2)]$p_0= p$.
  Then \eqref{eq.per.1} gives
 \begin{equation}\label{eq.per.4}
 \wt \Psi(z)= a_0 p+ a_1\wt p \left(p\frac{z-\wt q}{\wt p}+q\right)+ a_2 p z\left(p\frac{z-\wt q}{\wt p}+q\right)+\dots.
    \end{equation}
Setting again $z=0$ gives
\begin{equation}\label{eq.per.5}
 \wt \Psi(0)= a_0 p+ a_1\wt p \left(q-\frac{p\wt q}{\wt p}\right)= p+a_1 (\wt p-p).
    \end{equation}
    The expression in \eqref{eq.per.5} can take three values
    \begin{equation}\label{eq.per.6}
 \wt \Psi(0)= \left\{\begin{array}{cc}
 p   &\mbox{ if }a_1=0,\\
 \wt p &\mbox{ if }a_1=1,\\
 -\wt p +2p  &\mbox{ if }a_1=-1.
 \end{array}\right.
 \end{equation}
 
\item[Case 2.1)]If $\wt p\neq 2p$, then we see by \eqref{eq.per.6} that $|\Psi(0)|$ is bounded from below by a constant depending only on $p$ and $\wt p$.
 \item[Case 2.2)] $p_0=p$ and $\wt p=2p$. Observe that in this particular case
 
 \begin{equation}\label{eq.per.7}
 p\frac{z-\wt q}{\wt p}+q=\frac{1}{2}z+\frac{1}{2}.
   \end{equation}
 Hence inserting $z=-1$ into \eqref{eq.per.4} gives
    \begin{equation}\label{eq.per.8}
\wt \Psi(-1)= a_0p>0.
   \end{equation}
   Moreover, $\wt\Psi$ is continuous in some neighborhood of $z=-1$ (uniformly in $a_0,a_1,\dots$). Hence, there exists $\delta_1\in(0,1)$ depending only on $p$ such that $\vert\wt\Psi(1-\delta_1)\vert$ is bounded away from 0.
 
\end{itemize}

Again $\vert\wt \Psi(z)\vert$ for $\vert z\vert \leq 1$ is bounded in terms of a geometric series. Set $p^*=\frac{p}{\wt p}$ and $q^*=1-p^*$. Then
\begin{equation}\label{eq.per.9}
\vert\wt \Psi(z)\vert \leq \frac{1}{p^*(1-\vert z\vert)}.
    \end{equation}
     Now the claim of Lemma \ref{lem.4} follows by letting $\Omega$ be a ball contained in $\BB D$ such that $\partial\Omega$ is tangent to $\partial \D$ at 1 and $-1+\delta_1\in\Omega$, and then proceeding as in Lemma \ref{lem.2} using the subharmonic function $\log|\wt\Psi|$.
\end{proof}
    
\section{Proof of main result: Deletion probabilities varying with position}
\label{sec4}
  \begin{proof}[Proof of Theorem \ref{thm.main}]
  	Throughout the proof all constants should depend only on $\delta$ and $m$.
  Let ${\bf x}, {\bf y}\in[m]^n$ be two different strings. We first consider the case $m=2$. Let ${\bf a}={\bf x}-{\bf y}$. By Lemma \ref{lem.1},
  \begin{equation}\label{e.19}
 \sum_{j=0}^{n-1} \mathbb{E}\left(\widetilde X_j -\widetilde Y_j\right) w^j=\sum_{k=0}^{n-1} a_kp_k \prod_{\ell=0}^{k-1} (q_\ell+p_\ell w).
  \end{equation}
  Let $k^*= \min\{ k\,:\, a_k \neq 0\} $. Then the right side of \eqref{e.19} is equal to 
  \begin{equation}\label{e.20}
  \prod_{\ell<k^*} (q_\ell+p_\ell w)\left(p_{k^*}a_{k^*} +\sum_{k=k^*+1}^{n-1} a_kp_k \prod_{\ell=k^*}^{k-1} (q_\ell+p_\ell w)\right).
  \end{equation}
  
Let $L\in \mathbb{N}$. We first consider deletion probabilities satisfying Assumption (i). First, we verify that for $w\in \gamma_L$ 
  \begin{equation}\label{e.21}
 \left\vert \prod_{\ell<k^*} (q_\ell+p_\ell w) \right \vert \geq \exp(-nC_{2}/L^2),  
  \end{equation}
for some universal constant $C_{2}$.  To see that observe that by writing $w=\cos(\theta) + i \sin(\theta)$, for any $0<p<1$,
    \begin{eqnarray}\label{i.det10}
  \vert pw+(1-p)\vert^2&=& 1-2(1-p)p\left(1-\cos(\theta)\right) \nonumber\\&\geq& 1-(1-p)p\theta^2+ O(\theta^4)\nonumber\\
  &=&\exp\left( -(1-p)p\theta^2+O(\theta^4) \right),
  \end{eqnarray}
  where we used the series expansion of $\cos$. Since $\theta \in \gamma_L$, \eqref{e.21} follows. Define
  \begin{equation}\label{e.22}
A(w)=p_{k^*}a_{k^*} +\sum_{k=k^*+1}^{n-1} a_kp_k \prod_{\ell=k^*}^{k-1} (q_\ell+p_\ell w).
\end{equation}
By Lemma \ref{lem.2} there exists $w_1\in \; \gamma_L$ such that $\vert A(w_1)\vert \geq e^{-cL}$. Hence, taking absolute values in \eqref{e.19} gives for $w\in \gamma_L$,
\begin{eqnarray}\label{i.det1}
\sum_{j\geq 0} \left\vert\mathbb{E}\left(\widetilde X_j -\widetilde Y_j\right) \right\vert& \geq& \left\vert \prod_{\ell<k^*} (q_\ell+p_\ell w_1) \right \vert \vert A(w_1)\vert\\
&\geq& \exp(-nC_{2}/L^2) e^{-cL}.
\end{eqnarray}

To approximately maximize the term on the right side of \eqref{i.det1} we choose $L$ of order $n^{1/3}$ and obtain that there is a constant $C_{3}>0$ such that
\begin{equation}\label{i.det2}
\sum_{j\geq 0} \left\vert\mathbb{E}\left(\widetilde X_j -\widetilde Y_j\right) \right\vert \geq \exp\left(-C_{3}n^{1/3}\right).
\end{equation}
We now conclude the proof similarly as the proof of \cite[Theorem 1.1]{NaPe16} (see the argument starting with equation (2.4) of that paper). We first observe that for some $j\geq 0$ and $C_4>0$ we have $\left\vert\mathbb{E}\left(\widetilde X_j -\widetilde Y_j\right) \right\vert\geq \exp(-C_4n^{1/3})$. By Hoeffding's inequality this allows us to test whether our bit sequence is more likely to equal ${\bf x}$ or ${\bf y}$; in case our string equals either $\bf x$ or $\bf y$ our test fails with probability at most $\exp(-C_5n^{1/3})$. Repeating this for all possible pairs of bit sequences ${\bf x}, {\bf y}\in\{0,1 \}^n$, we can determine the original bit string with high probability as $n\rta\infty$. By increasing the constant $C$ appearing in the statement of the theorem the result holds for any $n$. See \cite{NaPe16} for a more detailed argument.

For $m\neq 2$ we proceed similarly. Let ${\bf x}, {\bf y}\in[m]^n$ for $m\neq 2$.  For each fixed $\zeta\in[m]$, we define $\wt {\bf x}$ and $\wt {\bf y}$ to be equal to ${\bf x}$ and ${\bf y}$, respectively, except that we replace $x_k$ by 1 if $x_k=\zeta$, and we replace $x_k$ by 0 if $x_k\neq \zeta$. Using the above procedure we can find all $k$ such that $x_k=\zeta$ in the original string. Repeating for all $\zeta\in[m]$ we determine the original string.

For deletion probabilities satisfying Assumption (ii) let 
\begin{equation}\label{e.22}
A(z)=p_{k^*}a_{k^*} +\sum_{k=k^*+1}^{n-1} a_kp_k \prod_{\ell=k^*}^{k-1} \left(\frac{p_\ell}{\wt p}z + q_\ell-\frac{p_\ell\wt q}{\wt p}\right).
\end{equation}
By Lemma \ref{lem.4} there exists $z_0\in \; \gamma_L$ such that $\vert A(z_0)\vert \geq e^{-cL}$.
Similarly as above, for $z\in \gamma_L$ and $w=\frac{z-\wt q}{\wt p}$,
  \begin{equation}\label{iii.det1}
 \left\vert \prod_{\ell<k^*} (q_\ell+p_\ell w) \right \vert \geq \exp(-nC_6/L^2),  
  \end{equation}
  and 
  \begin{equation}\label{iii.det4}
  \vert w\vert \geq \exp\left(-C_7/L^2\right).
  \end{equation}
 
Taking again absolute values in \eqref{e.19}, this gives for $z\in \gamma_L$, 
\begin{eqnarray}\label{ii.det3}
\sum_{j\geq 0} \left\vert\mathbb{E}\left(\widetilde X_j -\widetilde Y_j\right) \right\vert \vert w\vert^j & \geq& \left\vert \prod_{\ell<k^*} (q_\ell+p_\ell w) \right \vert \vert A(z)\vert \nonumber \\
&\geq& \exp(-nC_6/L^2) e^{-cL},
\end{eqnarray}
 Using \eqref{iii.det4}, this gives that $\left\vert\mathbb{E}\left(\widetilde X_j -\widetilde Y_j\right) \right\vert\geq \exp(-C_8n^{1/3})$ for some $j\geq 0$ and $C_8>0$. We conclude the proof as above. 
\end{proof}

\section{Proof of main result: Deletion probabilities varying with letter}

To prove Theorem \ref{thm1}, we first observe that the theorem is immediate from \cite{NaPe16} and \cite{DOS16} in the case where the deletion probabilities $q_0,\dots,q_{m-1}$ are known.  This follows since we can send the traces through a second deletion channel, where each letter $\zeta\in[m]$ is retained with probability $p_\zeta^{-1}\min_{\zeta'\in[m]} p_{\zeta'}$. The traces obtained in the second deletion channel can be obtained directly from ${\bf x}$ sent through a single deletion channel with constant retention probability $\min_{\zeta\in[m]} p_{\zeta}$, and ${\bf x}$ can therefore be reconstructed with $\exp( O(n^{1/3}) )$ traces. For the case of unknown deletion probabilities, we show in Lemma \ref{prop1} that we can obtain good estimates for the deletion probabilities by studying the traces. Then we use Lemma \ref{prop2} to argue that these estimates are sufficiently good, so that the single bit test still works when we use our estimated values for the deletion probabilities.

\begin{lemma}
	Consider the setting of Theorem \ref{thm1}, where we assume the deletion probabilities are unknown, and that each letter in $[m]$ appears at least once in ${\bf x}$. Given any $\delta,C_1>0$, we can find a $C_0>0$ depending only on $\delta$, $m$, and $C_1$, such that if we have at least $T=\lceil\exp(C_0 n^{1/3})\rceil$ traces, then we can use the traces to find an estimate $\wh p_\zeta$ for each $p_\zeta$ satisfying
	\eqb
	\P\left[ \max_{\zeta\in[m]} |\wh p_\zeta-p_\zeta|> \exp(-C_1 n^{1/3}) \right]<\delta.
	\label{eq4}
	\eqe
	\label{prop1}
\end{lemma}
\begin{proof}
	Fix $\zeta\in[m]$ and define $p:=p_\zeta$, $r:=|\{k\in[n]\,:\, x_k=\zeta\}|$, $\mu:=rp$, and $v:=rp(1-p)$. Observe that $1-p=v/\mu$. If $Y_t=|\{k\in[|\wt X^t|]\,:\, \wt X^t_k=\zeta \}|$, then $\E[Y_t]=\mu$ and $\Var[Y_t]=v$. We use the following estimates $\wh\mu,\wh v$, and $\wh p$ for $\mu,v$, and $p$, respectively,
	\eqbn
	\wh\mu = \frac 1T \sum_{t=1}^T Y_t,\qquad
	\wh v = \frac {1}{T-1} \sum_{t=1}^T(Y_t-\wh\mu)^2,\qquad
	1-\wh p=\wh v/\wh\mu.
	\eqen
	We have $\E[\wh\mu]=\mu$, $\E[\wh v]=v$, and $\Var[\wh\mu] = v/T<r/T$. Also observe that for a universal constant $C>0$, we have $\Var[\wh v] \leq Cr^4/T$. This means that for an appropriate $C_0>0$, with probability $1-o_n(1)$, we have $|\wh \mu-\mu|,|\wh v-v|<\exp(-C_1 n^{1/3})$. Upon increasing $C_0$ and using the definition of $\wh p$, this gives that $|\wh p-p|<\exp(-C_1 n^{1/3})$ with probability $1-o_n(1)$, which implies the lemma.
\end{proof}

For ${\bf x}=(x_0,\dots,x_{n-1})\in\R^n$ and ${\bf p}=(p_0,\dots,p_{m-1})\in(0,1]^m$, define
\begin{equation}
\Phi_{{\bf p}}^{\bf x}(w)= \sum_{k=0}^{n-1}x_kp_{x_k} \prod_{\ell=0}^{k-1}(p_{x_\ell} w+q_{x_\ell}).
\end{equation}
\begin{lemma}
	For any $C_2>0$, we can find a $C_1>0$ such that the following holds for any $p\in(0,1]$. Let ${\bf p}=(p,\dots,p)\in(0,1]^m$ and ${\bf{ p}'}=( p'_0,\dots, p'_{m-1})\in(0,1]^m$ satisfy
	\eqb
	|p-p'_\zeta|< \exp(-C_1 n^{1/3}),\qquad\forall \zeta\in[m].
	\label{eq3}
	\eqe
	Then for each ${\bf x}\in[m]^n$, each coefficient in the polynomial $\Phi^{\bf x}_{{\bf p}}(w)-\Phi^{\bf x}_{{\bf{p'}}}(w)$ has magnitude less than $\exp(-C_2 n^{1/3})$.
	\label{prop2}
\end{lemma}
\begin{proof}
	To simplify notation, let $\eps=\exp(-C_1 n^{1/3})$. The magnitude of the coefficient of $w^j$ in  $\Phi^{\bf x}_{{\bf p}}(w)-\Phi^{\bf x}_{{\bf{p'}}}(w)$, is bounded above by the following, with $B(k,p)$ denoting a binomial random variable 
	\eqbn
	m\sum_{k=0}^{n-1} \max\Big\{ 
	\big|\P[B(k,p+\wh\eps)=j] 
	-
	\P[B(k,p)=j] \big| 
	\,:\, 
	\wh\eps\in [ -\eps,\eps] 
	\Big\}+\eps,
	%\sum_{k=0}^{n-1} \max\Big\{ \big|(p+\wh\eps)^{j+1}(1-(p+\wh\eps))^{k} - p^{j+1}(1-p)^{k} \big| \,:\, \wh\eps\in [ -\eps,\eps] \Big\}.
	\eqen
	which is bounded by a polynomial multiple of $\eps$.
	%Since $|\frac{d}{dp}( p^{j+1}(1-p)^k )|\leq n$ for all $p\in(0,1)$, this is bounded by $n^2\eps$, which implies the lemma.
\end{proof}

\begin{proof}[Proof of Theorem \ref{thm1}]
	Throughout the proof we consider constants $C_0,C_1,C_2$ which may depend on $\delta$, $p^*:=\min_{\zeta\in[m]} p_\zeta$, and $m$, but which are independent of all other parameters.
	Let $T=\lceil\exp(C_0 n^{1/3})\rceil$ for some $C_0>0$ which will be determined later. Consider $T$ traces ${\bf{\wt X}}^{(1)},\dots,{\bf{\wt X}}^{(T)}$ obtained by sending ${\bf x}$ through the deletion channel considered in the statement of the theorem. Using the $T$ traces, we find an estimate $\wh p_\zeta$ for each $p_\zeta$ as described in Lemma \ref{prop1}, and let $\wh p^*=\min_{\zeta\in[m]} \wh p_\zeta$. 
	
 	Send each trace ${\bf{\wt X}}^{(t)}$ through a second deletion channel, so that we obtain traces ${\bf{\check X}}^{(1)},\dots,{\bf{\check X}}^{(T)}$. In the second deletion channel the letter $\zeta$ is retained with probability $\wh p^*/\wh p_\zeta$. Observe that each trace ${\bf{\check X}}^{(t)}$ can be obtained from ${\bf {x}}$ by considering a deletion channel in which the letter $\zeta$ is retained with probability $p^*_\zeta:=p_\zeta\wh p^*/\wh p_\zeta$. In particular, if our estimate $\wh p_\zeta$ for $p_\zeta$ is good for all $\zeta$, then each letter is retained with approximately the same probability $p^*$. Define ${\bf p}^*=(p^*_0,\dots,p^*_{m-1})\in(0,1]^m$ and ${\bf \wh p}^*=(\wh p^*,\dots,\wh p^*)\in(0,1]^m$; the first string represents the actual (unknown) retention probabilities, and while the second string represents our estimated retention probabilities.
	
	Given strings ${\bf y},{\bf y}'\in[m]^n$, let $\check {\bf Y}$ (resp.\ $\check {\bf Y}'$) denote a string obtained by sending $\bf y$ (resp.\ $\bf y'$) through the two deletion channels described above, i.e., the letter $\zeta$ is retained with probability $p^*_\zeta$. We first assume $m=2$. By Lemma \ref{lem.1},
	\eqb
	\begin{split}
	\sum_{j=0}^{n-1} \mathbb{E}&\left(\check Y_j -\check Y'_j\right) w^j
	=
	\Phi^{\bf y}_{{\bf { p^*}}}(w) - \Phi^{\bf y'}_{{\bf { p^*}}}(w)\\
	&=
	\Big(\Phi^{\bf y}_{{\bf {\wh p^*}}}(w) - \Phi^{\bf y'}_{{\bf {\wh p^*}}}(w)\Big) 
	+ 
	\Big(\Phi^{\bf y}_{{\bf { p^*}}}(w) - \Phi^{\bf y}_{{\bf {\wh p^*}}}(w)  \Big) 
	-
	\Big(\Phi^{\bf y'}_{{\bf { p^*}}}(w) - \Phi^{\bf y'}_{{\bf {\wh p^*}}}(w)  \Big). 
	\end{split}
	\label{eq5}
	\eqe
	Let 
	$$
	E(C_0,C_1) := \left\{ \max_{\zeta\in[m]} |\wh p^*-p^*_\zeta|\geq \exp(-C_1 n^{1/3}) \right\}.
	$$
	By \cite{NaPe16}, there is a $C_2>0$, such that for some $j\in[n]$, the absolute value of the coefficient of $w^j$ in $\Phi^{\bf y}_{{\bf {\wh p^*}}}(w) - \Phi^{\bf y'}_{{\bf {\wh p^*}}}(w)$ is at least $\exp(-C_2 n^{1/3}/2)$. From the proof in \cite{NaPe16} we see that $C_2$ is universal given any lower bound on $\wh p^*$, so on the event $E(C_0,C_1)^c$ and for sufficiently large $n$, we may assume that the constant $C_2$ depends only on $p^*$.
	
	Given $C_2>0$, define $C_1>0$ as in Lemma \ref{prop2}. By Lemma \ref{prop1}, we can find $C_0>0$ such that
	\eqb
	\P\left[ E(C_0,C_1) \right]<\delta/2.
	\label{eq6}
	\eqe
	By Lemma \ref{prop2} and \eqref{eq5}, on $E(C_0,C_1)^c$ and for all sufficiently large $n$, the absolute value of the coefficient of $w^j$ in the polynomial $\Phi^{\bf y}_{{\bf { p^*}}}(w) - \Phi^{\bf y'}_{{\bf { p^*}}}(w)$ is at least $\exp(-2C_2 n^{1/3})$ and of the same sign as the corresponding coefficient in $\Phi^{\bf y}_{{\bf {\wh p^*}}}(w) - \Phi^{\bf y'}_{{\bf {\wh p^*}}}(w)$. Therefore, on the event $E(C_0,C_1)^c$, Hoeffding's inequality gives that with probability at least $1-\exp( -T\exp(-4C_2 n^{1/3})/2 )$, we can determine if our unknown bit string $\bf x$ equals $\bf y$ or $\bf y'$, by considering $({\bf{\check X}}_j^{(t)})_{1\leq t\leq T}$. After possibly increasing $C_0$, we see by a union bound that on $E(C_0,C_1)^c$ and for $m=2$, we can identify ${\bf x}$ with probability at least $1-2^n\exp( -T\exp(-4C_2 n^{1/3})/2 )$.
	
	Given ${\y},{\y}'\in[m]^n$ for $m\neq 2$ we proceed similarly. For each fixed $\zeta\in[m]$ and with ${\bf x}=(x_0,\dots,x_{n-1})$, we first identify the set $A_\zeta:=\{ k\in[n]\,:\,x_k=\zeta \}$. Define $\wh \y$ and $\wh \y'$ to be equal to $\y$ and $\y'$, respectively, except that 
	all $x_k$ such that $x_k=\zeta$ has been replaced by 1, and 
	all $x_k$ such that $x_k\neq \zeta$ has been replaced by 0. Using the approach above with $\wh \y$ and $\wh \y'$, on the event $E(C_0,C_1)^c$ we can determine $A_\zeta$ except on an event of probability $1-2^n\exp( -T\exp(-4C_2 n^{1/3})/2 )$. We repeat the procedure for each $\zeta\in[m]$, and see that on the event $E(C_0,C_1)^c$ we can determine the sets $A_\zeta$, and hence ${\bf x}$, except on an event of probability $1-o_n(1)$. By increasing the constant $C_0$ if necessary, we can reconstruct the string with probability at least $1-\eps$ for any $n$.
\end{proof}

\bibliographystyle{abbrv}
\bibliography{literature}

\end{document}